\documentclass[12pt]{article}

\usepackage{amssymb}%
\usepackage{amsmath}%
\usepackage{amsthm}%

\usepackage{xcolor}%

\allowdisplaybreaks%

{\theoremstyle{plain}%
 \newtheorem{theorem}{Theorem}
 \newtheorem{corollary}{Corollary}
 \newtheorem{proposition}{Proposition}
}
{\theoremstyle{remark}
\newtheorem{remark}{Remark}
}
{\theoremstyle{definition}

\newtheorem{example}{Example}
}

\makeatletter
\newcommand{\dotDelta}{{\vphantom{\Delta}\mathpalette\d@tD@lta\relax}}
\newcommand{\d@tD@lta}[2]{%
 \ooalign{\hidewidth$\m@th#1\mkern-1mu\cdot$\hidewidth\cr$\m@th#1\Delta$\cr}%
}
\makeatother

\begin{document}

\begin{center}
{\Large Generalizations and variants of Knuth's old sum} 
\end{center}

\begin{center}
{\textsc{Arjun K. Rathie} \ \ \ \textsc{John M.\ Campbell} } 

 \ 

\end{center}

\begin{abstract}
 We extend the Reed Dawson identity for Knuth's old sum with a complex parameter, and we offer two separate hypergeometric series-based proofs of 
 this generalization, 
 and we apply this generalization to introduce binomial-harmonic sum identities. 
 We also provide another ${}_{2}F_{1}(2)$-generalization 
 of the Reed Dawson identity involving a free parameter. 
 We then apply Fourier--Legendre 
 theory to obtain an identity involving odd harmonic numbers that resembles the formula for Knuth's old sum, and the modified Abel lemma 
 on summation by parts is also applied. 
\end{abstract}

\section{Introduction}
 The sum given in the following identity is often referred to as 
 as Knuth's old sum \cite{Prodinger1994}: 
\begin{equation}\label{Knuthold}
 \sum_{k=0}^{n} \left( -\frac{1}{2} \right)^{k} \binom{n}{k} \binom{2k}{k} 
 = \begin{cases} 
 \frac{1}{2^n} \binom{n}{ n/2 } & \text{if $n$ is even}, \\ 
 0 & \text{if $n$ is odd}. 
 \end{cases} 
\end{equation}
 The above identity is often referred to as the \emph{Reed Dawson identity}.
 As in \cite{RathieKimParis2020}, we record that Reed Dawson introduced the above identity to Riordan, 
 with reference to the classic text \cite[p.\ 71]{Riordan1968}. 
 We also refer to the survey paper \cite{Prodinger1994} for a number 
 of different proofs of \eqref{Knuthold}. 
 Generalizations of \eqref{Knuthold} for the finite sums 
 $$ \sum_{k=0}^{2n} \left( -\frac{1}{2} \right)^{k} \binom{2n + \ell}{k + \ell} \binom{2k}{k} $$ and 
 $$ \sum_{k=0}^{2n+1} \left( -\frac{1}{2} \right)^{k} \binom{2n + 1 + \ell}{k + \ell} \binom{2k}{k} $$ 
 for $\ell \in \mathbb{N}_{0}$ are proved in \cite{RathieKimParis2020} 
 and natural extensions are established in \cite{KimMilovanovicParisRathie2021}. 
 This article is mainly inspired by the generalizations indicated below of Knuth's old sum identity. 

\begin{proposition}\label{mainproposition}
 For suitably bounded complex $\ell$, the following identity holds true: 
\begin{equation}\label{mainresult}
 \sum_{k=0}^{n} \left( -\frac{1}{2} \right)^{k} \binom{n + \ell}{k + \ell} \binom{2k + 2 \ell}{k} 
 = \begin{cases} 
 \frac{\binom{n + \ell}{\frac{n}{2}}}{2^n} & \text{if $n$ is even}, \\ 
 0 & \text{if $n$ is odd}. 
 \end{cases} 
\end{equation}
\end{proposition}

\begin{proposition}\label{202204161258AM1A}
 For suitably bounded complex $\ell$, the following identity holds true: 
\begin{equation}
 \sum_{k=0}^{n} \left( -\frac{1}{2} \right)^{k} 
 \frac{ \binom{n}{k} \binom{2k+2\ell}{k} }{\binom{k+\ell}{k}}
 = \begin{cases} 
 \frac{2^{-n} \binom{ n}{\frac{n}{2}}}{\binom{\frac{n}{2} + \ell}{\frac{n}{2}}} & \text{if $n$ is even}, \\ 
 0 & \text{if $n$ is odd}. 
 \end{cases} 
\end{equation}
\end{proposition}

 Although computer algebra systems such as Mathematica and Maple are able to evaluate the left-hand sides of the above 
 Propositions, we apply the above Propositions 
 to obtain results that cannot be handled by current CAS software. 

\subsection{Summary of main results}
 We offer two separate hypergeometric series-based proofs of 
 the generalization of the Reed Dawson identity indicated in Proposition \ref{mainproposition}: 
 One such proof employs hypergeometric identities obtained from Kummer's second formula 
 \cite[\S70]{Rainville1960}, and the latter such proof relies on a reindexing argument 
 together with the hypergeometric identity known as Gauss's second theorem. 

 Apart from the Reed Dawson-like harmonic sum identity shown in Example \ref{3Hk2H2k} that we may obtain in a straightforward way by applying 
 $\frac{d}{d\ell} \cdot \big|_{\ell = 0}$ to both sides of \eqref{mainresult}, proving the following identity 
 is much more difficult: 
\begin{equation*}
 \sum _{k=0}^m \frac{(-2)^k \binom{m}{k} H_k}{k+1} 
 = \begin{cases} 
 -\frac{2 }{m+1} O_{ \frac{m+1}{2}} & \text{if $m$ is odd,} \\ 
 0 & \text{if $m$ is even.} 
 \end{cases} 
\end{equation*}
 We are letting $O_{r} = 1 + \frac{1}{3} + \cdots + \frac{1}{2r-1}$ 
 denote the $r^{\text{th}}$ odd harmonic number. 
 Neither Maple 13 nor Maple 2020 is able to evaluate these harmonic sums. 

We apply, as in Section \ref{sectionFL} below, 
 building on recent results from \cite{CampbellLevrieNimbran2021}, 
 Fourier--Legendre (FL) theory to prove the binomial-harmonic sum identity 
\begin{equation}\label{oddKnuth}
 \sum_{k = 0}^{n} \left( -\frac{1}{4} \right)^{k} \binom{n}{k} \binom{2k}{k} O_{k} 
 = -\left( \frac{1}{4} \right)^{n} \binom{2n}{n} O_{n}, 
\end{equation}
 which closely resembles the formula for Knuth's old sum, as shown in \eqref{Knuthold}. 
 Again, neither Maple 13 nor Maple 2020 is able to correctly evaluate the above harmonic sum. 

 Finally, in Section \ref{sectionAbel}, 
 we conclude by briefly considering the application of the modified Abel lemma on summation by parts to 
 our results. 

\section{Hypergeometric proofs}
 We begin by introducing two separate proofs of   \eqref{mainresult} using classical hypergeometric infinite series. 

 \ 

\noindent \emph{First Proof of Proposition \ref{mainproposition}:} Let $S$ denote the left-hand side of 
 \eqref{mainresult}. We proceed to rewrite the binomial coefficients in the summand in 
 \eqref{mainresult} using the Pochhammer symbol. This gives us that
 $$ S = \frac{\Gamma\left( n + \ell + 1 \right)}{\Gamma\left( n + 
 1 \right)\Gamma\left( \ell + 1 \right)} \sum_{k = 0}^{n} \frac{ \left( -n \right)_{k} 
 \left( \ell + \frac{1}{2} \right)_{k} 2^{k} }{ k! \left( 2 \ell + 1 \right)_{k}}. $$
 Since Pochhammer symbols of the form $ \left( -n \right)_{k} $ vanish for $k > n$, 
 we find that: 
\begin{equation*}
 S = \frac{\Gamma\left( n + \ell + 1 \right)}{\Gamma\left( n + 
 1 \right)\Gamma\left( \ell + 1 \right)} \, {}_{2}F_{1}\!\!\left[ 
 \begin{matrix} 
 -n, \ell + \frac{1}{2} \vspace{1mm} \\ 
 2 \ell + 1
 \end{matrix} \ \Bigg| \ 2 \right]. 
\end{equation*}
 We now observe that the ${}_{2}F_{1}(2)$-function given above may be evaluated 
 via known hypergeometric identities 
 \cite[\S70]{Rainville1960}: For odd/even $n$, respectively, we have that
\begin{equation*}
 {}_{2}F_{1}\!\!\left[ 
 \begin{matrix} 
 -2 n, a \vspace{1mm} \\ 
 2 a 
 \end{matrix} \ \Bigg| \ 2 \right] = \frac{ \left( \frac{1}{2} \right)_{n} }{ \left( a + \frac{1}{2} \right)_{n} } 
\end{equation*}
 and that 
\begin{equation*}
 {}_{2}F_{1}\!\!\left[ 
 \begin{matrix} 
 -2 n, a \vspace{1mm} \\ 
 2 a + 1
 \end{matrix} \ \Bigg| \ 2 \right] = 0.
\end{equation*}
 This easily allows us to obtain the right-hand side of \eqref{mainresult}. \qed

 \ 

 On the other hand, Proposition \ref{mainproposition} may also be established by employing the classical hypergeometric identity known as \emph{Gauss's 
 second summation theorem} \cite{Rainville1960}: 
\begin{equation}\label{Gausssecond}
 {}_{2}F_{1}\!\!\left[ 
 \begin{matrix} 
 a, b \vspace{1mm} \\ 
 \frac{1}{2}(a + b + 1)
 \end{matrix} \ \Bigg| \ \frac{1}{2} \right] 
 = \frac{\Gamma\left( \frac{1}{2} \right) 
 \Gamma\left( \frac{1}{2} a + \frac{1}{2} b + 
 \frac{1}{2} \right) }{\Gamma\left( \frac{1}{2} a + 
 \frac{1}{2} \right) \Gamma\left( \frac{1}{2} b + \frac{1}{2} \right)}. 
\end{equation}

 \ 

\noindent \emph{Second Proof of Proposition \ref{mainproposition}:} We begin by letting $S$ denote the left-hand side of 
 \eqref{mainresult}. Let us apply a reindexing argument by reversing the indexing for 
 the finite sum on the left-hand side of \eqref{mainresult}, by replacing $k$ with $n - k$ in the summand under consideration. 
 So, we find that 
 $$ S = \sum_{k=0}^{n} (-1)^{n-k} 2^{-n + k} 
 \binom{n + \ell}{n - k + \ell} \binom{2n - 2 k + 2 \ell}{n - k}. $$
 Now, rewriting the binomial coefficients in the above summand using the Pochhammer symbol, we obtain that 
 $S$ may be written as: 
 $$ \frac{ (-1)^{n} 2^{n + 2 \ell} \Gamma\left( 
 n + \ell + \frac{1}{2} \right) \Gamma\left( n + \ell + 1 \right) }{ \sqrt{\pi} 
 \Gamma\left( n + 1 \right) \Gamma\left( n + 2 \ell + 1\right) } 
 \sum_{k=0}^{n} \frac{ \left( -n \right)_{k} 
 \left( -n - 2 \ell \right)_{k} 2^{-k} }{ \left( \frac{1}{2} - n - \ell \right)_{k} k!}. $$
 Since $ \left( -n \right)_{k} $ vanishes for $k > n$, we may rewrite the above expression as follows: 
 $$ \frac{ (-1)^{n} 2^{n + 2 \ell} \Gamma\left( 
 n + \ell + \frac{1}{2} \right) \Gamma\left( n + \ell + 1 \right) }{ \sqrt{\pi} 
 \Gamma\left( n + 1 \right) \Gamma\left( n + 2 \ell + 1\right) } 
 {}_{2}F_{1}\!\!\left[ 
 \begin{matrix} 
 -n, -n - 2 \ell \vspace{1mm} \\ 
 \frac{1}{2} - n - \ell 
 \end{matrix} \ \Bigg| \ \frac{1}{2} \right]. $$
 So, we find that the above ${}_{2}F_{1}\left( \frac{1}{2} \right)$-series may be evaluated via \eqref{Gausssecond}; 
 after some simplification, we obtain that 
 $$ S = 
 \frac{2^{n + 2 \ell} \Gamma\left( \frac{1}{2} - n \right) 
 \Gamma\left( n + \ell + 1 \right) }{ \Gamma\left( n + 1 \right) 
 \Gamma\left( n + 2 \ell + 1 \right) \Gamma\left( -\frac{1}{2} n + \frac{1}{2} \right) 
 \Gamma\left( \frac{1}{2} n - \ell + \frac{1}{2} \right)}, $$
 and this is easily seen to be equivalent to the desired result. \qed 

 \ 

\noindent \emph{Proof of Proposition \ref{202204161258AM1A}:} As before, we 
 convert binomial coefficients into Pochhammer symbols, so that the sum in 
 Proposition \ref{202204161258AM1A} may be written as 
 $$ {}_{2}F_{1}\!\!\left[ 
 \begin{matrix} 
 -n, \ell + \frac{1}{2} \vspace{1mm} \\ 
 2 \ell + 1 
 \end{matrix} \ \Bigg| \ 2 \right]. $$ 
 Using known results depending on the parity of $n$, the desired result easily follows. \qed 

\begin{remark}
 We also may prove the nonvanishing case of Proposition \ref{mainproposition} using the Wilf--Zeilberger (WZ) method 
 \cite{PetkovsekWilfZeilberger1996}; we assume familiarity with the WZ method. For the even case, we set 
\begin{equation}\label{mainFnk}
 F(n, k) = {\frac {{2\,n+\ell\choose k+\ell}{2\,k+2\,\ell\choose k}{2}^{2\,n}}{{2\,n+\ell
\choose n}} \left( -{\frac{1}{2}} \right) ^{k}}, 
\end{equation}
 and the Maple command {\tt WZMethod} may be used to determine the following WZ proof certificate: 
 $$ -{\frac {k \left( k+2\,\ell \right) }{ \left( -2\,n+k-1 \right) \left( -
2\,n+k-2 \right) }}. $$ 
 Similarly, Zeilberger's algorithm \cite[\S6]{PetkovsekWilfZeilberger1996} may be applied to the vanishing case for 
 Proposition \ref{mainproposition}. As for the latter Proposition, a WZ proof certificate for the nonvanishing case 
 is recorded below: 
 $$ -\frac{k (k+2 \ell)}{(k-2 n-2) (k-2 n-1)}. $$
 Again, Zeilberger's algorithm applies to the vanishing case. 

\end{remark}

\section{Binomial-harmonic summations}
 As in our above generalization of the Reed Dawson identity, 
 we are letting $\ell$ denote a suitably bounded complex parameter. This allows us to obtain interesting results on binomial-harmonic sums, by 
 applying differential operators to both sides of \eqref{mainresult}, and by making use of the series expansion \cite[\S9]{Rainville1960} 
\begin{equation}\label{psidef}
 \psi(z) = \frac{d}{dz}\ln\Gamma(z)=\frac{\Gamma'(z)}{\Gamma(z)} =-\gamma+\sum_{n=0}^{\infty}\frac{z-1}{(n+1)(n+z)} 
\end{equation}
 for the digamma function, where
 the Euler--Mascheroni constant is such that $ \gamma = \lim_{n\to\infty}\big(H_n-\ln n\big)$. 

\begin{example}\label{3Hk2H2k}
 By applying $\frac{d}{d\ell} \cdot \big|_{\ell = 0}$, the above approach gives us that 
\begin{equation}\label{equation3Hk}
 \sum _{k=0}^{2 n} \left(-\frac{1}{2}\right)^k \binom{2 k}{k} \binom{2 n}{k} \left(3 H_k-2 H_{2 k}\right)
 = \left( \frac{1}{4} \right)^{n} \binom{2n}{n} H_{n} 
\end{equation}
 for all $n \in \mathbb{N}_{0}$, letting $H_{n} = 1 + \frac{1}{2} + \cdots + \frac{1}{n}$ denote 
 the $n^{\text{th}}$ harmonic number. 
\end{example}

 Notice the close resemblance of the binomial-harmonic identity in Example \ref{3Hk2H2k} to both the identity for Knuth's old sum and the formula highlighted 
 in \eqref{oddKnuth}. By applying $\frac{d}{d\ell} \cdot \big|_{\ell = \frac{1}{2}}$ to both sides of \eqref{mainresult}, we may use the resultant identity 
 to obtain another remarkable result, as below, but the derivation/proof of the summation identity shown below is much more intricate compared to 
 \eqref{equation3Hk}, so we find it appropriate to highlight the below result as a Corollary to our generalization of the Reed Dawson identity. It 
   does not seem to be possible to the usual moment/integral formulas for the sequence of 
 harmonic numbers to prove the following Corollary. 

\begin{corollary}\label{mCorollary}
 The identity 
\begin{equation*}
 \sum _{k=0}^m \frac{(-2)^k \binom{m}{k} H_k}{k+1} 
 = \begin{cases} 
 -\frac{2 }{m+1} O_{ \frac{m+1}{2}} & \text{if $m$ is odd} \\ 
 0 & \text{if $m$ is even} 
 \end{cases} 
\end{equation*}
 holds for all $m \in \mathbb{N}_{0}$. 
\end{corollary}

\begin{proof}
 After the application of $\frac{d}{d\ell} \cdot \big|_{\ell = \frac{1}{2}}$ in 
 the above Proposition, and after much simplification and rearrangement using 
 the digamma expansion in \eqref{psidef}, we obtain that 
 $$ -4 \sum _{k=0}^{2 n} \frac{\left(-\frac{1}{2}\right)^k \binom{2 n}{k} 4^k}{(k+1)^2}-2 \sum _{k=0}^{2 n} \frac{\left(-\frac{1}{2}\right)^k \binom{2 n}{k}
 4^k H_{k} }{k + 1} $$
 equals $\frac{2 \left(H_{n+1}-2 H_{2 n+2}\right)}{2 n+1}$ for each member $n$ of $\mathbb{N}_{0}$. 
 So, it remains to evaluate the ${}_{3}F_{2}(2)$-series 
 $$ \sum _{k=0}^{2 n} \frac{\left(-\frac{1}{2}\right)^k \binom{2 n}{k} 4^k}{(k+1)^2}
 = {}_{3}F_{2}\!\!\left[ 
 \begin{matrix} 
 1, 1, -2n \vspace{1mm} \\ 
 2, 2
 \end{matrix} \ \Bigg| \ 2 \right], $$ 
 noting that Mathematica 13, the current version of Mathematica in 2022, 
 is not able to evaluate this ${}_{3}F_{2}(2)$-expression for a free parameter $n$. Since 
 $$ \sum _{k=0}^{2 n} \frac{\left(-\frac{1}{2}\right)^k \binom{2 n}{k} 4^k x^k}{k+1}
 = \frac{2 x (1-2 x)^{2 n}-(1-2 x)^{2 n}+1}{2 (2 n+1) x}, $$ 
 we are led to evaluate the antiderivative of the right-hand side of the above equality, as below: 
\begin{align*}
 & \frac{1}{2 (2 n+1)^2} \Bigg( (1-2 x)^{2 n+1} \, 
 {}_{2}F_{1}\!\!\left[ 
 \begin{matrix} 
 1, 2n + 1 \vspace{1mm} \\ 
 2 n + 2 
 \end{matrix} \ \Bigg| \ 1 - 2 x \right] - \\ 
 & (1-2 x)^{2 n+1}+2 n \ln (x)+\ln (x) \Bigg). 
\end{align*}
 Setting $x \to 1$ and $x \to 0$, this can be used to prove the even case for the Corollary under consideration. 
 So, since we have proved that 
\begin{equation*}
 \sum _{k=0}^{2 n} \frac{(-2)^k \binom{2 n}{k} H_k}{k+1} 
\end{equation*}
 always vanishes, we proceed to make use of the 
 binomial identity 
$$ \frac{(2 n-k+1) \binom{2 n+1}{k}}{2 n+1}=\binom{2 n}{k}, $$
 and we are led to find that 
 $$ \sum _{k=0}^{2 n} \binom{2 n+1}{k} H_k (-2)^k \left(-1+\frac{2}{k+1}+\frac{2 n}{k+1}\right) $$
 always vanishes. 
 So, it remains to evaluate 
 $$ - \sum _{k=0}^{2 n} \binom{2 n+1}{k} H_k (-2)^k. $$
 So, by evaluating 
$$ \sum _{k=0}^{2 n} \binom{1+2 n}{k} \left(-k x^{k-1} \ln (1-x)\right) (-2)^k (-1)$$ 
 as 
$$ -(-2)^{2 n+1} (2 n+1) x^{2 n} \ln (1-x)-2 (2 n+1) (1-2 x)^{2 n} \ln (1-x) $$
 and by expressing the antiderivative of the above evaluation as the product 
 of $-\frac{1}{2 (n+1)}$ and 
\begin{align*}
 & (-2)^{2 n+1} x^{2 n+2} 
 \, {}_{2}F_{1}\!\!\left[ 
 \begin{matrix} 
 1, 2 (n+1) \vspace{1mm} \\ 
 2 n+3
 \end{matrix} \ \Bigg| \ x \right] + \\ 
 & (1-2 x)^{2 (n+1)} 
 \, {}_{2}F_{1}\!\!\left[ 
 \begin{matrix} 
 1, 2 n+2 \vspace{1mm} \\ 
 2 n+3
 \end{matrix} \ \Bigg| \ 2 x-1 \right] + \\
 & 2 (n+1) \left((-2)^{2 n+1} x^{2 n+1}+2 x (1-2
 x)^{2 n}-(1-2 x)^{2 n}\right) \ln (1-x), 
\end{align*}
 and by taking the required limits as $x \to 1$ and $x \to 0$, we find that 
 $$ \sum _{k=0}^{2 n} \frac{(-2)^k \binom{2 n+1}{k} H_k}{k+1}=\frac{\left(4^n-1\right) H_{2 n}}{n+1}+\frac{H_n}{2 (n+1)}+\frac{4^n-1}{(n+1) (2 n+1)} $$
 for all $n \in \mathbb{N}_{0}$. 
 This is easily seen to be equivalent to: 
 $$ \sum _{k=0}^{2 n+1} \frac{(-2)^k \binom{2 n+1}{k} H_k}{k+1} = -\frac{ O_{n+1} }{n + 1}$$
 for all $n \in \mathbb{N}_{0}$. 
\end{proof}

 We remark that the identity $$ \sum_{k=0}^{2n} (-1)^k \binom{2n}{k} \binom{2n+k}{k} \binom{2k}{k} 4^{2n-k} H_{k} = \binom{2n}{n}^2 H_{2n}, $$ 
 which was intrdocued and proved by Tauraso via the WZ method 
 in \cite{Tauraso2018}, and the above identity was quite recently reproduced in 
 \cite{Chen2022} and considered in the context of the study of Ramanujan-like series for $\frac{1}{\pi}$. 
 This inspires the application of the results introduced in this article in relation 
 to the material in \cite{Chen2022,Tauraso2018}. 

 By applying appropriate differential operators to Proposition \ref{202204161258AM1A}, 
 we may again obtain the binomial-harmonic identities introduced above.

\section{A Fourier--Legendre-based approach}\label{sectionFL}
 We proceed to prove the remarkable identity highlighted in \eqref{oddKnuth}. The second author had independently discovered and proved this identity 
 in relation to the research in \cite{CampbellLevrieNimbran2021} on a trilogarithmic extension of Catalan's constant. For the sake of brevity, we assume 
 some basic familiarity with Legendre polynomials. As indicated in \cite{CampbellLevrieNimbran2021}, a conjectured evaluation for the exotic 
 hypergeometric series 
\begin{equation*}
 \sum_{n=0}^{\infty} \frac{ \binom{2n}{n}^{2} }{2^{4n} (2n+1)^2 } = {}_{4}F_{3}\!\!\left[ 
 \begin{matrix} 
 \frac{1}{2}, \frac{1}{2}, \frac{1}{2}, \frac{1}{2} \vspace{1mm} \\ 
 1, \frac{3}{2}, \frac{3}{2} 
 \end{matrix} \ \Bigg| \ 1 \right] 
\end{equation*}
 was offered in 2015 \cite{Nimbran2015}, and an FL-based proof was offered in 2019 in \cite{CantariniDAurizio2019}, with reference to the identity 
\begin{equation}\label{FLlnsqrt}
 \int_{0}^{1} \frac{\ln(x)}{\sqrt{x}} P_{n}(2x-1) 
 \, dx = \frac{4 (-1)^n H_{n}}{2n+1} - \frac{8(-1)^n H_{2n}}{2n+1} - \frac{4(-1)^n}{(2n+1)^2}, 
\end{equation}
 letting $P_{n}(2x-1)$ denote the shifted Legendre polynomial of order $n$. 
 The integration identity given above may be proved using 
 using the moment formula
\begin{equation}\label{momentLegendre}
 \int_{0}^{1} x^{p} P_{n}(2x-1) \, dx = 
 \frac{\Gamma^{2}(p+1)}{\Gamma(p - n + 1) \Gamma(p + n + 2)}, 
\end{equation}
 by applying 
 $ \frac{\partial}{\partial p} \cdot \Big|_{p = -\frac{1}{2}}$ 
 to both sides of \eqref{momentLegendre}, recalling the series expansion 
 for the digamma function shown in \eqref{psidef}. This leads us toward the below proof. 

\begin{theorem}
 The binomial-harmonic analogue identity in \eqref{oddKnuth} holds true. 
\end{theorem}

\begin{proof}
 This follows in a direct way by replacing the shifted Legendre polynomial in the integrand in \eqref{FLlnsqrt} with the corresponding finite sum 
 $$ P_{n}(2x-1) = \sum_{k=0}^{n} \binom{n}{k}^{2} (x-1)^{n-k} x^{k} $$ and by then integrating term-by-term. 
\end{proof}

\section{Conclusion}\label{sectionAbel}
 We conclude by briefly considering how 
 a bijective result known as the \emph{modified Abel lemma on summation 
 by parts} (cf.\ \cite{ChenChen2020,Wang2018,WangChu2018}) may be applied in relation to our results. 
 This bijective identity is such that 
\begin{equation}\label{modifiedAbel}
 \sum_{i = 1}^{\infty} B_{i} \nabla A_{i} = \left( 
 \lim_{m \to \infty} A_{m} B_{m + 1} \right) - A_{0} B_{1} + \sum_{i=1}^{\infty} A_{i} \dotDelta B_{i} 
\end{equation}
 if the limit shown above exists and if one of the two infinite summations shown above converges, letting $\nabla$ and $\dotDelta$ be such that 
 $ \nabla \tau_{n} = \tau_{n} - \tau_{n-1}$ and $ \dotDelta \tau_{n} = \tau_{n} - \tau_{n + 1}$. 
 For example, setting $$ A_{i} := 
 -\frac{(i-n) (-n)_i}{n \Gamma (i+1)} $$ and 
 $$ B_{i} = \frac{2^i \left(\ell+\frac{1}{2}\right)_i}{(2 \ell+1)_i}, $$ 
 this gives us the Reed Dawson-like identity shown below, using our first Proposition: 
\begin{equation*}
 \sum _{k=0}^n 
 \left(-\frac{1}{2}\right)^k \binom{n+\ell}{k+\ell} \binom{2 k+2 \ell}{k} 
 \frac{ k (n - k) }{k+2 \ell+1} 
 = \begin{cases} 
 -2^{-n} n \binom{n + \ell}{\frac{n}{2}} & \text{if $n$ is even}, \\ 
 0 & \text{otherwise}. 
 \end{cases}
\end{equation*}
 Similarly, by setting 
 $$ A_{k} := \frac{(-1)^k (k+\ell+1) \binom{\ell+n}{k+\ell+1}+\ell \binom{\ell+n}{\ell}}{\ell+n} $$ 
 and $$ B_{k} := \left(\frac{1}{2}\right)^k \binom{2 k+2 l}{k}, $$ 
 we may obtain that 
 the identity 
\begin{align}
 & \sum_{k=0}^{n} \left(-\frac{1}{2}\right)^k \binom{2 k}{k} \binom{n}{k} 
 \frac{ (2 k+1) \left(k^2+3 k+3\right) (n - k) }{(k+1)^2 (k+2) (k+3)} \nonumber \\ 
 & = \begin{cases} 
 \frac{1}{2}-\frac{\binom{n}{\frac{n}{2}} (n+1)}{2^n (n+2)} 
 & \text{if $n$ is even}, \\ 
 \frac{1}{2} & \text{if $n$ is odd}. \nonumber
 \end{cases} 
\end{align}
 holds true for $n \in \mathbb{N}_{0}$.

 \ 

\noindent Arjun K. Rathie

\noindent Department of Mathematics, Vedant College of Engineering and Technology

\noindent (Rajasthan Technical University), Bundi 323021, Rajasthan, India

\noindent {\tt akrathie@gmail.com}

 \ 

\noindent John M.\ Campbell

\noindent Department of Mathematics and Statistics, York University

\noindent Toronto, Ontario, Canada

\noindent {\tt jmaxwellcampbell@gmail.com}

\end{document}